\documentclass[11pt,a4paper]{article}

\usepackage[utf8]{inputenc}
\usepackage[T1]{fontenc} 
\usepackage[english]{babel}

\usepackage{geometry}
\usepackage{amsmath}
\usepackage{amsfonts}
\usepackage{amssymb}
\usepackage{amsthm}
\usepackage{mathtools} 
\usepackage{mathrsfs}  
\usepackage{relsize}
\usepackage{bigints}
\usepackage{mathtools}

\usepackage{graphicx}
\usepackage{float}
\usepackage{subfig}
\usepackage{wrapfig}
\usepackage{array}
\usepackage{multirow}
\usepackage{diagbox}
\usepackage{booktabs} 
\usepackage{algorithm}
\usepackage{algpseudocode}
\usepackage[toc,page]{appendix}

\usepackage{cite}
\usepackage[fixlanguage]{babelbib}
\usepackage{url}
\usepackage[dvipsnames]{xcolor}
\usepackage[unicode,
            psdextra,
            colorlinks=true,
            linkcolor=blue,
            citecolor=green!60!black, 
            urlcolor=blue
           ]{hyperref}
           
\usepackage[nameinlink,capitalize,noabbrev]{cleveref}

\makeindex

\theoremstyle{plain} 
\newtheorem{thm}{Theorem}[section]
\newtheorem{prop}[thm]{Proposition}
\newtheorem{corollary}[thm]{Corollary}  

\theoremstyle{definition} 
\newtheorem{definition}[thm]{Definition}

\theoremstyle{remark} 

\newcommand{\R}{\mathbb{R}}

\providecommand{\keywords}[1]{\vspace{3mm}\noindent\textbf{Keywords: } #1}
\providecommand{\classifications}[1]{\noindent\textbf{Mathematics Subject Classification (2020):} #1}

\title{
Data-Driven Diffusion Processes on Differential Forms via the Projected Ambient Connection Laplacian
}

\author{
Alvaro Almeida Gomez\\
Centro de Modelamiento Matemático (CNRS IRL 2807),\\
Universidad de Chile,\\
Beaucheff 851, Santiago, Chile\\
\texttt{alvaroalmeidagomez182@gmail.com}
\and
Jorge Duque Franco\\
Dirección de Investigación, Vicerrectoría Académica,\\
Instituto de Matemáticas, Universidad de Talca,\\
Casilla 721, Talca, Chile\\
\texttt{jorge.duque@utalca.cl}
}

\date{\today}

\begin{document}

\maketitle

\begin{abstract}
We develop a data-driven approximation of the projected ambient connection Laplacian acting on differential forms over smooth Riemannian manifolds sampled by point clouds. The proposed construction extends the classical framework of diffusion maps and Vector Diffusion Maps from scalar functions and tangent vector fields to differential forms of arbitrary degree.

Our approach is based on a novel representation of differential forms as alternating differential arrays obtained through an extension of the classical musical isomorphism. This representation enables the construction of a matrix-valued diffusion operator that approximates the projected ambient connection Laplacian directly from point cloud data without requiring a mesh or simplicial complex. 

The proposed discretization admits the asymptotically optimal kernel bandwidth scaling inherited from diffusion maps, leading to sharper convergence guarantees than previous data-driven approximations of the Hodge Laplacian. Building upon this operator, we derive a fully data-driven explicit Euler scheme for the heat equation on differential forms and validate the proposed methodology through numerical experiments on the unit sphere. The experiments confirm the predicted  decay of the analytical solution and demonstrate the effectiveness of the proposed discretization.

The proposed framework provides a natural generalization of Vector Diffusion Maps to differential forms of arbitrary degree and establishes a practical foundation for the numerical approximation of geometric partial differential equations directly from point cloud data.
\end{abstract}

\keywords{
Diffusion Maps,
Differential Forms,
Connection Laplacian,
Heat Equation,
Point Clouds,
Manifold Learning
}

\classifications{58J35, 58J50, 53C21,65D18,35K08.}

\section{Introduction}

Differential forms constitute one of the fundamental mathematical tools for
describing geometric and physical phenomena on smooth manifolds.
They provide a coordinate-free framework for formulating conservation laws,
electromagnetism, fluid dynamics, and geometric partial differential equations.
Among the most important differential operators acting on differential forms
is the Hodge Laplacian, whose spectral and analytical properties encode both
the topology and the geometry of the underlying manifold
\cite{rosenberg1997laplacian,jost2011riemannian}.

The numerical approximation of differential operators on manifolds has received
considerable attention over the last two decades.
Classical approaches rely on finite element exterior calculus,
discrete exterior calculus (DEC),
or finite difference methods defined over triangulated manifolds.
Finite element exterior calculus provides a mathematically rigorous framework
for discretizing the de Rham complex while preserving its cohomological
structure \cite{Arnold2006}.
Similarly, discrete exterior calculus constructs discrete analogues of
differential forms and exterior derivatives on simplicial complexes,
preserving many geometric identities exactly
\cite{Desbrun2005,Hirani2003}.
Although these approaches possess excellent theoretical properties,
they require the availability of a mesh or simplicial decomposition,
which becomes increasingly difficult to construct for high-dimensional data
or manifolds represented only through point clouds.

The emergence of manifold learning has motivated the development of
mesh-free approximations of differential operators directly from sampled data.
Among the most influential techniques are diffusion maps,
introduced by Coifman and Lafon \cite{coifman2006diffusion},
which recover the Laplace--Beltrami operator through kernel integral operators.
Subsequent convergence analyses established rigorous asymptotic error estimates
for graph Laplacians constructed from random samples
\cite{singer2006graph,Hein2007}.
These developments have led to numerous data-driven numerical methods
for approximating intrinsic geometric quantities without requiring an explicit
triangulation of the manifold.

Beyond scalar-valued functions, diffusion-based methods have also been extended
to vector-valued data on manifolds. A prominent example is the framework of
Vector Diffusion Maps (VDM), introduced by Singer and Wu
\cite{singer2012vector}, which provides a data-driven approximation of the
connection Laplacian acting on tangent vector fields. The key idea is to combine
local diffusion with orthogonal alignments between neighboring tangent spaces,
thereby capturing the intrinsic geometry of the underlying manifold. The
resulting operator has become a fundamental tool in manifold learning,
geometric data analysis, and shape processing, with notable applications
including the reconstruction of three-dimensional molecular structures in
cryo-electron microscopy (Cryo-EM) \cite{singer2011three}.
Nevertheless, existing VDM methodologies are fundamentally restricted to
vector fields, corresponding to differential $1$-forms under the classical
musical isomorphism.

The numerical treatment of higher-order differential forms directly from
point clouds remains considerably less developed.
Although discrete exterior calculus naturally generalizes to arbitrary
degrees, its dependence on simplicial complexes limits its applicability in
many data-driven settings.
Recently, several mesh-free approaches have been proposed for approximating
the exterior derivative and Hodge Laplacian on differential forms using local
kernel constructions and point cloud data
\cite{Belkin2008,almeida6959520local}.
These methods demonstrate that differential forms can be approximated directly
from sampled manifolds while preserving important geometric properties.

The objective of the present work is to extend the diffusion maps philosophy
from vector fields to differential forms of arbitrary degree.
Our starting point is a novel representation of differential forms through
alternating differential arrays introduced in
\cite{almeida6959520local}.
Using an extension of the classical musical isomorphism,
every differential $k$-form is identified with an alternating tensor-valued
array embedded in the ambient Euclidean space.
This representation allows differential forms to be manipulated using standard
linear algebra while preserving their intrinsic geometric meaning.

Building upon this representation, we introduce a data-driven approximation of
the projected ambient connection Laplacian acting on differential forms.
Unlike the intrinsic connection Laplacian, the ambient operator acts
componentwise in the embedding space, making it particularly suitable for
kernel-based approximations. By projecting this operator onto the bundle of
alternating differential arrays, we obtain a well-defined operator on
differential forms. By the Weitzenb\"ock formula, this operator differs from
the Hodge Laplacian only by explicit intrinsic and extrinsic zero-order
curvature terms.

An important advantage of the proposed projected ambient connection Laplacian
over the data-driven approximation of the Hodge Laplacian introduced in
\cite{almeida6959520local} is that our construction admits the asymptotically
optimal bandwidth parameter
\[
t = N^{-2/(d+6)},
\]
where $N$ denotes the number of sample points. This scaling is optimal for
uniformly distributed samples, as established in \cite{singer2006graph}, and
enables a substantially sharper theoretical convergence analysis. In contrast,
the approximation error obtained in \cite{almeida6959520local} is of order
\[
\mathcal{O}\!\left(\frac{1}{\sqrt{\log\log N}}\right).
\]
The optimal choice of the bandwidth therefore yields a significantly improved
theoretical approximation error while preserving the intrinsic geometric
structure of the underlying manifold.

The resulting discretization naturally extends the framework of Vector
Diffusion Maps (VDM). For arbitrary
degrees $k$, however, it defines a matrix-valued diffusion operator acting on
alternating differential arrays, thereby providing a unified data-driven
framework for approximating the projected ambient connection Laplacian on
differential forms of arbitrary degree.

The main contributions of this paper are summarized as follows.

\begin{itemize}
\item We introduce a computational representation of differential
forms as alternating differential arrays through an extension of the
musical isomorphism.

\item We construct a projected ambient connection Laplacian acting on
differential arrays and establish its relationship with the Hodge Laplacian.

\item We derive a fully data-driven matrix approximation of this operator
directly from point cloud data using diffusion kernels.

\item We develop an explicit Euler scheme for the corresponding heat equation
and validate the proposed discretization through numerical experiments on the
unit sphere.
\end{itemize}

The remainder of the paper is organized as follows.
Section~2 introduces differential arrays and extends the musical
isomorphism to arbitrary differential forms.
Section~3 develops the projected ambient connection Laplacian and its
data-driven matrix representation.
Section~4 presents the numerical discretization of the associated heat
equation together with computational experiments.
Finally, Section~5 concludes the paper and discusses future research
directions.

\section{Differential Arrays}
\label{formasdiferenciables}

Throughout this section, we work in the following geometric setting:
$\mathcal{M}$ is a compact, boundaryless Riemannian submanifold of
$\mathbb{R}^n$, where the inner product on each tangent space
$T_x\mathcal{M}$ is induced by the standard inner product of $\mathbb{R}^n$.

In this section, we introduce the notion of \emph{data-driven differential
arrays}. Concretely, given a differential $k$-form $\omega$ on $\mathcal{M}$
and a finite set of sample points $x_1, x_2, \ldots, x_N \in \mathcal{M}$,
we address the following question: since differential forms are inherently
abstract objects, how can one construct a concrete computational
representation of $\omega$ at the sample points?

To address this question, we introduce \emph{$k$-dimensional arrays over
$\mathbb{R}^n$}, which generalize vectors ($k = 1$) and matrices ($k = 2$).

\begin{definition}
\label{def:k-array}
A \emph{$k$-dimensional array} in $\mathbb{R}^n$ is a function
\begin{equation*}
    f \colon \underbrace{I_n \times I_n \times \cdots \times I_n}_{k\ \text{times}}
    \to \mathbb{R},
\end{equation*}
where $I_n \coloneqq \{1, 2, \ldots, n\}$.
The set of all $k$-dimensional arrays in $\mathbb{R}^n$ forms a vector space
under pointwise addition and scalar multiplication.
\end{definition}

Our goal is to associate with each differential $k$-form $\omega$ on
$\mathcal{M}$ a data object, defined at the sample points, that faithfully
encodes the action of $\omega$. More precisely, for each sample point
$x_i \in \mathcal{M}$, we seek a $k$-dimensional array
$W_i \colon I_n^k \to \mathbb{R}$ that represents the action of
$\omega_{x_i}$ on tangent vectors to $\mathcal{M}$. Our approach extends
the classical musical isomorphism from alternating $1$-forms to alternating
$k$-forms for arbitrary $k \geq 1$.

We recall that, for any linear subspace $V \subseteq \mathbb{R}^n$ and any
linear functional $\ell \in V^{*}$, there exists a unique vector $W \in V$
such that
\begin{equation*}
    \ell(v) = \langle W, v \rangle \qquad \text{for all } v \in V,
\end{equation*}
where $\langle \cdot, \cdot \rangle$ denotes the standard inner product on
$\mathbb{R}^n$. Applying this fiberwise to each tangent space
$T_x\mathcal{M} \subseteq \mathbb{R}^n$, every differential $1$-form on
$\mathcal{M}$ is canonically identified with a vector field on $\mathcal{M}$.
This identification is the classical musical isomorphism for $k = 1$, and
our construction extends it to all $k \geq 1$.

Given vectors $v_1, v_2, \ldots, v_k \in \mathbb{R}^n$, their
\emph{tensor product} $v_1 \otimes v_2 \otimes \cdots \otimes v_k$ is the
$k$-dimensional array defined by
\begin{equation*}
    (v_1 \otimes v_2 \otimes \cdots \otimes v_k)(i_1, i_2, \ldots, i_k)
    \coloneqq
    v_1(i_1)\,v_2(i_2)\cdots v_k(i_k),
\end{equation*}
for every $(i_1,\ldots,i_k)\in I_n^k$.

For a linear subspace $V \subseteq \mathbb{R}^n$, we denote by
\begin{equation*}
    V^{\otimes k}
    \coloneqq
    \underbrace{V \otimes V \otimes \cdots \otimes V}_{k\ \mathrm{times}}
\end{equation*}
the linear span of all elementary tensors
$v_1 \otimes v_2 \otimes \cdots \otimes v_k$ with
$v_1,\ldots,v_k \in V$. We equip $V^{\otimes k}$ with the
\emph{Frobenius inner product}, defined by
\begin{equation*}
    \langle A, B \rangle_F
    \coloneqq
    \sum_{i_1,\ldots,i_k \in I_n}
    A(i_1,\ldots,i_k)\,B(i_1,\ldots,i_k),
\end{equation*}
for all $A,B \in V^{\otimes k}$.
The subspace of \emph{alternating arrays}, denoted by
$\Lambda^k V \subseteq V^{\otimes k}$, is defined as follows.

\begin{definition}
\label{def:alternating-array}
A $k$-dimensional array $W \colon I_n^k \to \mathbb{R}$ is called
\emph{alternating} if
\[
    W(i_{\sigma(1)}, i_{\sigma(2)}, \ldots, i_{\sigma(k)})
    = \operatorname{sgn}(\sigma)\,
    W(i_1, i_2, \ldots, i_k)
\]
for every permutation $\sigma \in S_k$ and every choice of indices
$i_1, \ldots, i_k \in I_n$. The collection of all alternating
$k$-dimensional arrays forms a linear subspace of $V^{\otimes k}$,
which we denote by $\Lambda^k V$.
\end{definition}

The space $\Lambda^k V$ of alternating $k$-forms, endowed with the Frobenius
inner product inherited from $V^{\otimes k}$, is an inner product space.
Moreover, if $\{v_1,v_2,\ldots,v_n\}$ is an orthonormal basis of $V$, then the family
\[
\left\{
v_{i_1}\wedge v_{i_2}\wedge\cdots\wedge v_{i_k}
:\;
1\le i_1<i_2<\cdots<i_k\le n
\right\}
\]
is an orthonormal basis of $\Lambda^k V$. Here,
\[
v_{i_1}\wedge v_{i_2}\wedge\cdots\wedge v_{i_k}
=
\frac{1}{\sqrt{k!}}
\sum_{\sigma\in S_k}
(\operatorname{sgn}\sigma)\,
v_{i_{\sigma(1)}}\otimes
v_{i_{\sigma(2)}}\otimes
\cdots\otimes
v_{i_{\sigma(k)}},
\]
where the sum is taken over all permutations $\sigma\in S_k$. We now extend the musical isomorphisms to alternating $k$-forms.
\begin{prop}
\label{prop:musical-generalization}
Let $V \subseteq \mathbb{R}^n$ be a linear subspace and let
$\omega \in \Lambda^k(V^*)$ be an alternating $k$-linear form on $V$.
Then there exists a unique alternating array $W \in \Lambda^k V$ such that
\begin{equation}
\label{eq:musical-iso}
    \omega(v_1,\dots,v_k)
    = \langle W,\, v_1 \otimes \cdots \otimes v_k \rangle_F
    \qquad
    \text{for all } v_1,\dots,v_k \in V.
\end{equation}
\end{prop}

The proof is given in Appendix~\ref{proofprop1}. Applying
Proposition~\ref{prop:musical-generalization} to $V = T_x\mathcal{M}$ at
each point $x \in \mathcal{M}$ yields the following corollary.

\begin{corollary}
\label{cor:musical-manifold}
Let $\mathcal{M} \subseteq \mathbb{R}^n$ be a Riemannian submanifold
endowed with the metric induced by the Euclidean inner product, and let
$\omega$ be a smooth differential $k$-form on $\mathcal{M}$. For every
point $x \in \mathcal{M}$, there exists a unique alternating array
$W_x \in \Lambda^k T_x\mathcal{M}$ such that
\begin{equation}
\label{eq:musical-pointwise}
    \omega_x(v_1,\dots,v_k)
    = \bigl\langle W_x,\, v_1 \otimes \cdots \otimes v_k
      \bigr\rangle_F
\end{equation}
for all $v_1,\dots,v_k \in T_x\mathcal{M}$, where
$\langle \cdot,\cdot\rangle_F$ is the Frobenius inner product on
$(T_x\mathcal{M})^{\otimes k}$ induced by the Euclidean metric.
Moreover, the section $x \mapsto W_x$ is smooth.
\end{corollary}

\begin{definition}[$k$-differential array]
\label{def:k-diff-array}
A \emph{$k$-differential array} on $\mathcal{M}$ is a smooth section of the
bundle $\Lambda^k T\mathcal{M}$ of alternating $k$-arrays, i.e.\ an element
of $\Gamma(\Lambda^k T\mathcal{M})$.
\end{definition}

By Corollary~\ref{cor:musical-manifold}, every differential $k$-form
$\omega$ on $\mathcal{M}$ determines a unique $k$-differential array. This
pointwise correspondence defines a map
\begin{equation}
\label{eq:sharp-map}
    \sharp \colon \Omega^k(\mathcal{M})
    \longrightarrow \Gamma\!\left(\Lambda^k T\mathcal{M}\right),
    \qquad
    \omega \longmapsto W,
\end{equation}
which we call the \emph{musical isomorphism} between differential $k$-forms
and $k$-differential arrays. Since Proposition~\ref{prop:musical-generalization}
provides a fiberwise isomorphism at each $x \in \mathcal{M}$, the map
$\sharp$ is a vector-bundle isomorphism.

Throughout this paper, we identify differential $k$-forms with their
associated $k$-differential arrays via the musical isomorphism $\sharp$.
This identification simplifies the notation and allows all local
computations to be carried out directly in terms of differential arrays.
Accordingly, all subsequent constructions are formulated for
$k$-differential arrays while systematically exploiting this
identification.

The space $\Lambda^k T\mathcal{M}$ of $k$-differential arrays is endowed
with the $L^2$ inner product
\[
\langle W_1, W_2\rangle
=
\int_{\mathcal{M}}
\langle W_1(x), W_2(x)\rangle_F \, d\mathrm{Vol}(x),
\]
for any $W_1,W_2\in\Lambda^k T\mathcal{M}$, where
$\langle\cdot,\cdot\rangle_F$ denotes the Frobenius inner product on the
fiber and $d\mathrm{Vol}(x)$ is the Riemannian volume form on
$\mathcal{M}$.

\section{ Ambient Connection Laplacian}
\label{sec:infinitesimal_approx_connection_laplacian}

Recall from Section~\ref{formasdiferenciables} that a differential $k$-array is a smooth map $\omega \colon \mathcal{M} \to \mathbb{R}^{n^{k}}$ such that, for every $x \in \mathcal{M}$, the value $\omega(x)$ belongs to the tensor subspace $(T_x\mathcal{M})^{\otimes k} \subset (\mathbb{R}^n)^{\otimes k} \cong \mathbb{R}^{n^k}$. Here, the inclusion is induced by the differential of the isometric embedding $\mathcal{M} \hookrightarrow \mathbb{R}^n$. Via the metric-induced musical isomorphisms introduced in Section~\ref{formasdiferenciables}, these arrays are naturally identified with differential $k$-forms.

To study these differential arrays within the ambient space $\mathbb{R}^{n^k}$, we introduce the \emph{ambient connection Laplacian}. This operator acts on sections of the trivial bundle $E = \mathcal{M} \times \mathbb{R}^{n^k}$, which is equipped with the flat connection $\bar{\nabla}$ inherited from the standard Euclidean structure on $\mathbb{R}^{n^k}$. For a smooth section $f \in \Gamma(E) \cong C^{\infty}(\mathcal{M}, \mathbb{R}^{n^k})$, its ambient connection Laplacian is defined as the metric trace of its second covariant derivative:
\begin{equation}
    \Delta_{\bar{\nabla}} f := \operatorname{tr}_{g}(\bar{\nabla}^{2} f).
    \label{eq:ambient_laplacian}
\end{equation}
Here, $\bar{\nabla}^{2} f \in \Gamma(T^{*}\mathcal{M} \otimes T^{*}\mathcal{M} \otimes E)$ is evaluated with respect to both the Levi-Civita connection $\nabla$ on $\mathcal{M}$ and the flat connection $\bar{\nabla}$ on $E$.

Because the metric on $\mathcal{M}$ is induced by the ambient Euclidean inner product and $\bar{\nabla}$ is the flat Euclidean connection, this ambient connection Laplacian acts strictly componentwise. Consequently, for any smooth map $f = (f_{1}, \ldots, f_{n^{k}}) \colon \mathcal{M} \to \mathbb{R}^{n^{k}}$, we have
\[
    \Delta_{\bar{\nabla}}f = \left( \Delta_{\mathcal{M}}f_{1}, \ldots, \Delta_{\mathcal{M}}f_{n^{k}} \right),
\]
where $\Delta_{\mathcal{M}}$ denotes the scalar Laplace--Beltrami operator on $\mathcal{M}$. 

This resulting Laplacian elegantly encodes both the intrinsic geometry of $\mathcal{M}$ and the extrinsic geometry of the embedding $\mathcal{M} \hookrightarrow \mathbb{R}^n$. Note that while $f(x)$ may lie entirely within the tangent spaces $(T_x\mathcal{M})^{\otimes k}$, the evaluation of $\Delta_{\bar{\nabla}}f$ will generally produce vectors with non-trivial normal components relative to the manifold, explicitly reflecting its extrinsic nature.

We now derive an infinitesimal approximation of the ambient connection Laplacian acting on vector-valued functions. The construction is based on a diffusion-type integral operator associated with a smooth sampling density $q$ on the manifold $\mathcal{M}$. This formulation is particularly important because it naturally leads to a data-driven approximation of the ambient connection Laplacian acting on differential arrays.

Let $G_t:\mathcal{M}\times\mathcal{M}\to\mathbb{R}$ denote the Gaussian kernel
\[ G_t(x,y) = \exp\left( -\frac{|x-y|^2}{2t^2} \right), \]
where $|\cdot|$ denotes the Euclidean norm in the ambient space. Assume that $q\in C^\infty(\mathcal{M})$ is a strictly positive probability density on $\mathcal{M}$. We define the normalization function
\begin{equation} \label{eq:normalization_function}
d_t(x) = \int_{\mathcal{M}} G_t(x,y) q(y) \, d\mathrm{Vol}(y),
\end{equation}
where $d\mathrm{Vol}$ denotes the Riemannian volume measure on $\mathcal{M}$.

Using this normalization, we define the integral operator $\mathbf{P}_t$ acting on smooth vector-valued functions $f:\mathcal{M}\to\mathbb{R}^{n^k}$ by
\begin{equation} \label{operadorP}
\mathbf{P}_t f(x) = \frac{1}{d_t(x)} \int_{\mathcal{M}} G_t(x,y) f(y) q(y) \, d\mathrm{Vol}(y),
\end{equation}
where the integral is understood componentwise. The associated infinitesimal generator is defined as follows.

\begin{definition}[Infinitesimal Generator] \label{def:infinitesimal_generator}
Let $\mathbf{P}_t$ be the operator defined in \eqref{operadorP}. The infinitesimal generator $L_t$ is given by
\[ L_t f(x) := \frac{f(x)-\mathbf{P}_t f(x)}{t^2}, \]
for every smooth vector-valued function $f:\mathcal{M}\to\mathbb{R}^{n^k}$.
\end{definition}

The operator $L_t$ can be viewed as a vector-valued extension of the infinitesimal generator arising in the diffusion maps framework. Indeed, since $\mathbf{P}_t$ acts componentwise on the coordinates of $f$, the asymptotic analysis developed for scalar diffusion operators applies independently to each component. Consequently, by the classical diffusion maps expansion of Coifman and Lafon \cite{coifman2006diffusion}, one obtains
\[ \lim_{t\to 0^+} L_t f(x) = \frac{1}{2}\Delta_{\nabla}f(x), \]
where $\Delta_{\nabla}$ denotes the ambient connection Laplacian. Therefore, $\mathbf{P}_t$ provides a diffusion semigroup whose infinitesimal generator converges to the ambient connection Laplacian in the limit $t\to 0$.

This characterization immediately suggests a discrete approximation. Let $\{x_1,\ldots,x_N\} \subset \mathcal{M}$ be independent samples drawn according to the density $q$. By the law of large numbers, the integral operator $\mathbf{P}_t$ can be approximated by the discrete operator
\[ \mathbf{P}^{\mathrm{dis}}_t f(x_i) = \frac{1}{d_t^{\mathrm{dis}}(x_i)} \sum_{j=1}^{N} G_t(x_i,x_j) f(x_j), \]
where
\[ d_t^{\mathrm{dis}}(x_i) = \sum_{j=1}^{N} G_t(x_i,x_j). \]

The corresponding discrete infinitesimal generator is given by
\begin{equation}
\label{discreteinfgen}
 L_t^{\mathrm{dis}}f(x_i) = \frac{f(x_i)-\mathbf{P}^{\mathrm{dis}}_t f(x_i)}{t^2}.   
\end{equation}

A particularly important case arises when the samples are drawn from a uniform distribution on the manifold, that is, when $q(x) = 1 / \mathrm{Vol}(\mathcal{M})$. In this setting, Singer \cite{singer2006graph} established the pointwise convergence rate for the discrete generator. Letting $d$ denote the intrinsic dimension of the manifold $\mathcal{M}$, the convergence rate is given by
\begin{equation}
\label{convergencerate}
    L_t^{\mathrm{dis}}f(x_i) = \frac{1}{2}\Delta_{\nabla}f(x_i) + \mathcal{O}\left(\frac{1}{\sqrt{N} t^{(d+2)/4}}\right) + \mathcal{O}(t).
\end{equation}

Balancing the variance and bias error terms leads to the optimal choice for the bandwidth parameter $t$, as derived in \cite{singer2006graph}:
\begin{equation}
\label{parametroinfinite}
    t = \frac{1}{N^{2/(d+6)}}.
\end{equation}

\subsection{Projected Ambient Connection Laplacian on Differential Arrays}
\label{sec:conexiondifferentialforms}

Recall that for a smooth differential $k$-array
$\omega \colon \mathcal{M}\to\mathbb{R}^{n^k}$, the ambient connection
Laplacian defines a smooth map
$\Delta_{\nabla}\omega\colon\mathcal{M}\to\mathbb{R}^{n^k}$.
In general, however, $\Delta_{\nabla}\omega$ does not preserve the space of
differential $k$-arrays, and hence it is not an endomorphism of this bundle.

To overcome this difficulty, we project the ambient operator onto the bundle
$\Lambda^kT\mathcal{M}$. More precisely, for every
$x\in\mathcal{M}$, we define the \emph{projected ambient connection Laplacian}
by
\begin{equation}
\label{proyeLaplac}
(\tilde{\Delta}_{\nabla}\omega)(x)
:=
\mathbf{P}_{\Lambda^kT_x\mathcal{M}}
\bigl(\Delta_{\nabla}\omega(x)\bigr),
\end{equation}
where
$\mathbf{P}_{\Lambda^kT_x\mathcal{M}}$
denotes the orthogonal projection onto
$\Lambda^kT_x\mathcal{M}$.

By construction, $\tilde{\Delta}_{\nabla}$ is a well-defined operator on
differential $k$-arrays. Nevertheless, it should not be identified with the
intrinsic connection Laplacian $\nabla^{*}\nabla$ acting on differential
forms. Indeed, if $M^d\hookrightarrow\mathbb{R}^N$ is an isometrically embedded
Riemannian manifold, then differentiation in the ambient Euclidean space
followed by orthogonal projection onto the tangent bundle introduces an
additional zero-order operator determined by the second fundamental form of
the embedding. More precisely,
\[
\tilde{\Delta}_{\nabla}
=
\nabla^{*}\nabla-\mathcal{A},
\]
where
\[
\mathcal{A}
=
\sum_{\alpha=1}^{N-d}S_{\alpha}^{\,2},
\]
and $S_{\alpha}$ denotes the shape operator associated with the $\alpha$-th
unit normal vector field. Therefore, the projected ambient connection
Laplacian differs from the intrinsic connection Laplacian only by an
extrinsic zero-order curvature operator.

This identity naturally relates $\tilde{\Delta}_{\nabla}$ to the
Hodge--de~Rham Laplacian. By the classical Weitzenb\"ock formula
\cite{rosenberg1997laplacian,jost2011riemannian},
\[
\Delta_H=\nabla^{*}\nabla+\mathcal{R},
\]
where $\mathcal{R}$ is the intrinsic Weitzenb\"ock curvature endomorphism.
Combining both identities yields the exact relation
\[
\boxed{
\Delta_H
=
\tilde{\Delta}_{\nabla}
+\mathcal{A}
+\mathcal{R},
}
\]
Hence, the projected ambient connection Laplacian and the Hodge Laplacian
differ exclusively by explicit zero-order operators, namely the intrinsic
curvature term $\mathcal{R}$ and the extrinsic contribution $\mathcal{A}$
induced by the embedding.

Consequently, combining the diffusion-map approximation developed in the
previous section with the projected operator above yields a discrete operator
whose orthogonal projection is well defined on differential arrays and,
through the musical isomorphisms, approximates the Hodge Laplacian up to an
explicit zero-order perturbation determined by the intrinsic and extrinsic
geometry of the manifold.

\section{Matrix-Based Computations}

Building upon the discretization of the infinitesimal generator of the ambient
connection Laplacian $\Delta_{\nabla}$ introduced in
\eqref{discreteinfgen}, together with the discussion in
Section~\ref{sec:conexiondifferentialforms} concerning its action on
differential forms, we now derive a matrix-based, data-driven representation
of the projected ambient connection Laplacian. Recall that the ambient
connection Laplacian is projected onto the space of differential
$k$-forms according to
\begin{equation}
    (\tilde{\Delta}_{\nabla}\omega)(x)
    :=
    \mathbf{P}_{\Lambda^k T_x\mathcal{M}}
    \bigl(\Delta_{\nabla}\omega(x)\bigr),
\end{equation}
where
$\mathbf{P}_{\Lambda^k T_x\mathcal{M}}$
denotes the orthogonal projection onto
$\Lambda^kT_x\mathcal{M}$. This projection is realized through the
musical isomorphisms and defines a smooth linear operator on the space
$\Omega^k(\mathcal{M})$ of differential $k$-forms. Furthermore, as
discussed in Section~\ref{sec:conexiondifferentialforms}, the projected
ambient connection Laplacian $\tilde{\Delta}_{\nabla}$ can be expressed as
the Hodge Laplacian together with a  perturbation term.

The objective of this section is to construct a local, data-driven matrix
representation of the projected operator
$\tilde{\Delta}_{\nabla}$ using only point-cloud data. This matrix
formulation provides a practical computational framework for approximating the
action of the projected ambient connection Laplacian on differential forms from
discrete samples of the underlying manifold.

To this end, we first introduce a matrix representation of differential arrays,
which, through the musical isomorphisms, are naturally identified with
differential forms. This identification enables us to represent differential
forms as matrix-valued objects while preserving their underlying geometric
structure. Building upon this representation, we derive a corresponding
matrix formulation of the projected ambient connection Laplacian
$\tilde{\Delta}_{\nabla}$. Throughout, the equivalence between the matrix
representation and the action of $\tilde{\Delta}_{\nabla}$ on differential forms
is understood via the musical isomorphisms.

The following data-driven construction of differential arrays was originally
introduced in \cite{almeida6959520local}. For completeness, we briefly recall
this construction, as it provides the foundation for the matrix-based
representation and computation of the projected ambient connection Laplacian.

\subsection{Matrix representation of differential arrays}
\label{sec:matrixdifarrays}

We begin by constructing a data-driven representation  of
$k$-differential arrays $\Gamma(\Lambda^k T\mathcal{M})$ from a finite collection of
$N$ sample points $X = \{x_1, x_2, \dots, x_N\}$. These points are assumed to
be realizations of $N$ independent and identically distributed (i.i.d.)
random variables $X_1, X_2, \dots, X_N$, drawn from a smooth probability
density $q(\cdot)$ supported on a $d$-dimensional compact Riemannian manifold
$\mathcal{M}$. 

For that, we begin with the local construction of
$k$-differential arrays. Given a $k$-differential array $W$, its evaluation
at a sample point $x_i$ can be written as
\begin{equation}
\label{localW}
   W(x_i) = \sum_{J} f_{J}(x_i)\, O_{J}(x_i),
\end{equation}
where the sum ranges over all ordered $k$-tuples $J = (j_1, j_2, \dots, j_k)$
with $1 \le j_1 < j_2 < \cdots < j_k \le d$. Here, $f_{J}(x_i)$ are
real-valued coefficient functions, and $O_J(x_i)$ denotes an element of the
orthonormal basis of $\Lambda^k T_{x_i}\mathcal{M}$ (see Section \ref{formasdiferenciables}),
defined via the wedge product
\begin{equation}
\label{definicionnuevaO}
    O_{J}(x_i) = O_{j_1}(x_i) \wedge \cdots \wedge O_{j_k}(x_i),
\end{equation}
where $\{O_{1}(x_i), \dots, O_{d}(x_i)\}$ is an orthonormal basis for the
tangent space $T_{x_i}\mathcal{M}$. In the proposed methodology, we assume
that orthonormal bases for the tangent spaces $\{T_{x_i}\mathcal{M}\}_{i=1}^{N}$
and the intrinsic dimension $d$ of the manifold $\mathcal{M}$ are either
known \emph{a priori}, or estimated via local principal component analysis
(PCA). In the latter case, we follow the approach
of~\cite{singer2012vector, singer2011orientability}, which constructs an
orthonormal basis for each tangent space $T_{x_i}\mathcal{M}$, together with
the maximum likelihood estimator of~\cite{levina2004maximum} for the
intrinsic dimension $d$. We stress, however, that the contribution of the
present section lies in exploiting these orthonormal bases to derive a
matrix formulation that approximates the Ambient Connection Laplacian. The
estimation of the tangent spaces themselves is a preprocessing step and is
not the focus of this work.

Using the local representation in \eqref{localW}, we discretize a
$k$-differential array $W \in \Gamma(\Lambda^k T\mathcal{M})$ by
representing it in matrix form as
\[
\mathbf{O}_k * \mathbf{f},
\]
where the product $*$ denotes the standard matrix multiplication. The matrices
$\mathbf{O}_k$ and $\mathbf{f}$ are defined as follows.\\
\textbf{Definition of $\mathbf{f}$:} The matrix $\mathbf{f}$ consists of $N$
blocks, each of size $\binom{d}{k}\times 1$. The $i$-th block is given by  
\begin{equation*}
    \mathbf{f}(i)=\begin{bmatrix}
f_{J_1}(x_i)\\
f_{J_2}(x_i)\\
\vdots\\
f_{J_{\binom{d}{k}}}(x_i)
\end{bmatrix},
\end{equation*}
where the multi-indices $J_1, J_2, \dots, J_{\binom{d}{k}}$ correspond to all
possible $k$-tuples
$$J_l=(j^l_1, \dots, j^l_k), \qquad 1 \le j^l_1 < \dots < j^l_k \le d.$$ 
Thus, the full matrix $\mathbf{f}$ has size $\binom{d}{k} N\times 1$.\\
\textbf{Definition of $\mathbf{O}_k$:} The matrix $\mathbf{O}_k$ consists of
$N\times N$ blocks, each of size $n^k\times \binom{d}{k}$. The block at
position $(i,j)$ is defined as
\begin{equation}
\label{def:negritaO}
    \mathbf{O}_k(i,j)= \begin{cases}
			\overline{O}_{k}(i)  & \text{if } i=j, \\
            0_{ n^k\times \binom{d}{k}}, & \text{if } i \neq j,
		 \end{cases}
\end{equation}
for $i,j\in\{1,\dots,N\}$, so that $\mathbf{O}_k$ is block-diagonal. Here,
$\overline{O}_{k}(i)$ is the matrix
\begin{equation*}
   \overline{O}_{k}(i)=\begin{bmatrix} O_{J_1}(x_i) & O_{J_2}(x_i) & \cdots & O_{J_{\binom{d}{k}}}(x_i)
   \end{bmatrix},
\end{equation*}
where each $O_{J_l}(x_i)$ is defined as in \cref{definicionnuevaO} and is
regarded as a column vector embedded in $\R^{n^k}$. Overall, $\mathbf{O}_k$
has size $n^k N\times \binom{d}{k} N$. The values of the $k$-differential
array $W(x_i)$ correspond to the $i$-th block of the product
$\mathbf{O}_k * \mathbf{f}$.

\subsection{Matrix representation of the projected discrete infinitesimal generator}
\label{matrixrepresentacionC}

We now derive a matrix representation of the projected approximation of the
ambient connection Laplacian
$(\tilde{\Delta}_{\nabla}\omega)(x)$ introduced in
Equation~\eqref{proyeLaplac}, namely,
\begin{equation*}
    (\tilde{\Delta}_{\nabla}W)(x)
    :=
    \mathbf{P}_{\Lambda^k T_x\mathcal{M}}
    \bigl(\Delta_{\nabla}W(x)\bigr),
\end{equation*}
where
$\mathbf{P}_{\Lambda^k T_x\mathcal{M}}$
denotes the orthogonal projection onto
$\Lambda^kT_x\mathcal{M}$. The construction is based on the observation that the projected
infinitesimal generator
\[
2*\mathbf{P}_{\Lambda^k T_x\mathcal{M}}
\bigl(L_t W(x)\bigr),
\]
where $L_t$ is the infinitesimal generator introduced in
Definition~\ref{def:infinitesimal_generator},
approximates the projected connection Laplacian
\[
\mathbf{P}_{\Lambda^k T_x\mathcal{M}}
\bigl(\Delta_{\nabla}W(x)\bigr),
\]
as established in
Section~\ref{sec:infinitesimal_approx_connection_laplacian}.

Throughout this subsection, let
$\{x_1,\ldots,x_N\}\subset\mathcal{M}$
be independent samples drawn from a smooth probability density
$q$ on $\mathcal{M}$.
Our goal is to construct a matrix representation of the  data-driven discrete
infinitesimal generator
$L_t^{\mathrm{dis}}$, of $L_t^t$ introduced in
Equation~\eqref{discreteinfgen}, and subsequently obtain a matrix
approximation of the projected ambient connection Laplacian defined over differential forms.

First, we assume that $W$ is a differential array described as equation \eqref{localW},  and we can write its evaluation at $x_i$ as
\begin{equation*}
   W(x_i) = \sum_{J} f_{J}(x_i)\, O_{J}(x_i),
\end{equation*}
where the sum ranges over all ordered $k$-tuples $J = (j_1, j_2, \dots, j_k)$
with $1 \le j_1 < j_2 < \cdots < j_k \le d$. and, $f_{J}(x_i)$ are
real-valued coefficient functions, and $O_J(x_i)$ denotes the wedge product
 via the wedge product
\begin{equation}
    O_{J}(x_i) = O_{j_1}(x_i) \wedge \cdots \wedge O_{j_k}(x_i),
\end{equation}
where $\{O_{1}(x_i), \dots, O_{d}(x_i)\}$ is an orthonormal basis for the
tangent space $T_{x_i}\mathcal{M}$.

Using this notation, the discrete infinitesimal generator can be written as
\begin{equation*}
    L_t^{\mathrm{dis}}W(x_i)
    =
    \frac{1}{t^2\,d_t^{\mathrm{dis}}(x_i)}
    \sum_{j=1}^{N}
    G_t(x_i,x_j)\bigl(W(x_i)-W(x_j)\bigr),
\end{equation*}
where
\begin{equation*}
    d_t^{\mathrm{dis}}(x_i)
    :=
    \sum_{j=1}^{N}
    G_t(x_i,x_j).
\end{equation*}

Therefore, the projection of the discrete infinitesimal generator onto
$\Lambda^k T_{x_i}\mathcal{M}$ is given by
\begin{equation}
\label{eqproye1}
    \mathbf{P}_{\Lambda^k T_{x_i}\mathcal{M}}
    \bigl(L_t^{\mathrm{dis}}W(x_i)\bigr)
    =
    \frac{1}{t^2\,d_t^{\mathrm{dis}}(x_i)}
    \sum_{j=1}^{N}
    G_t(x_i,x_j)\,
    \mathbf{P}_{\Lambda^k T_{x_i}\mathcal{M}}
    \bigl(W(x_i)-W(x_j)\bigr).
\end{equation}

Now suppose that $W$ is represented locally as in Equation~\eqref{localW}. Since
$\{O_L(x_i)\}_L$ forms an orthonormal basis of
$\Lambda^k T_{x_i}\mathcal{M}$, the orthogonal projection of
$W(x_j)$ onto $\Lambda^k T_{x_i}\mathcal{M}$ is given by
\begin{equation}\label{eqproye2}
    \mathbf{P}_{\Lambda^k T_{x_i}\mathcal{M}}
    \bigl(W(x_j)\bigr)
    =
    \sum_{L}
    \left\langle
        W(x_j),\, O_L(x_i)
    \right\rangle_{F}
    O_L(x_i),
\end{equation}
where the sum is taken over all multi-indices
$L=(l_1,l_2,\ldots,l_k)$ satisfying
\[
1 \le l_1 < l_2 < \cdots < l_k \le d.
\]

On the other hand, using the local representation \eqref{localW}, we obtain
\begin{equation*}
    \left\langle
        W(x_j),\, O_L(x_i)
    \right\rangle_{F}
    =
    \sum_{J}
    f_J(x_j)\,
    \left\langle
        O_J(x_j),\, O_L(x_i)
    \right\rangle_{F},
\end{equation*}
where the sum is taken over all multi-indices
\[
J=(j_1,\ldots,j_k),
\qquad
1\le j_1<\cdots<j_k\le d.
\]
Furthermore, by the definition of the Frobenius inner product on alternating arrays,
\begin{equation*}
\left\langle
O_J(x_j),\, O_L(x_i)
\right\rangle_{F}
=
\frac{1}{k!}
\sum_{\alpha,\beta\in S_k}
\operatorname{sgn}(\alpha)\,
\operatorname{sgn}(\beta)
\prod_{r=1}^{k}
\left\langle
O_{j_{\alpha(r)}}(x_j),
O_{l_{\beta(r)}}(x_i)
\right\rangle,
\end{equation*}
where $S_k$ denotes the symmetric group on $\{1,\ldots,k\}$. For convenience, define the matrix
\[
A^{L,J}(i,j)
=
\left(
\left\langle
O_{l_p}(x_i),
O_{j_q}(x_j)
\right\rangle
\right)_{p,q=1}^{k}.
\]
Then
\begin{equation*}
\left\langle
O_J(x_j),\, O_L(x_i)
\right\rangle_{F}
=
\frac{1}{k!}
\sum_{\alpha,\beta\in S_k}
\operatorname{sgn}(\alpha)\,
\operatorname{sgn}(\beta)
\prod_{r=1}^{k}
A^{L,J}_{\beta(r),\alpha(r)}(i,j).
\end{equation*}
Introducing the permutation
\[
\gamma=\beta\circ\alpha^{-1},
\]
and using the multiplicativity of the signature,
\[
\operatorname{sgn}(\beta)
=
\operatorname{sgn}(\gamma)\operatorname{sgn}(\alpha),
\]
we obtain
\begin{align*}
\left\langle
O_J(x_j),\, O_L(x_i)
\right\rangle_{F}
&=
\frac{1}{k!}
\sum_{\alpha,\gamma\in S_k}
\operatorname{sgn}(\gamma)
\prod_{r=1}^{k}
A^{L,J}_{\gamma(r),r}(i,j)\\
&=
\sum_{\gamma\in S_k}
\operatorname{sgn}(\gamma)
\prod_{r=1}^{k}
A^{L,J}_{\gamma(r),r}(i,j)\\
&=
\det\!\bigl(A^{L,J}(i,j)\bigr).
\end{align*}
The second equality follows because the summand is independent of $\alpha$, so the summation over $\alpha$ contributes a factor of $k!$, which cancels the prefactor $1/k!$. The final equality is precisely the Leibniz formula for the determinant. Consequently,
\begin{equation}
\label{expprod}
\left\langle
W(x_j),\, O_L(x_i)
\right\rangle_{F}
=
\sum_{J}
f_J(x_j)\,
\det\!\bigl(A^{L,J}(i,j)\bigr).
\end{equation}

Observe that
\[
A^{L,J}(i,j)
=
\bigl(O^{L}(x_i)\bigr)^T
O^{J}(x_j),
\]
where, for every multi-index
\[
M=(m_1,\ldots,m_k),
\]
the matrix
\[
O^{M}(x_l)\in\mathbb{R}^{d\times k}
\]
is the submatrix of
\[
O(x_l)
=
\bigl[
O_1(x_l),\ldots,O_d(x_l)
\bigr]
\]
formed by the columns indexed by $m_1,\ldots,m_k$, namely,
\[
O^{M}(x_l)
=
\bigl[
O_{m_1}(x_l),\,
O_{m_2}(x_l),\,
\ldots,\,
O_{m_k}(x_l)
\bigr].
\]
Combining \eqref{eqproye1}, \eqref{eqproye2}, and \eqref{expprod} yields
\begin{equation}
\label{expansiongrande}
\begin{aligned}
\mathbf{P}_{\Lambda^kT_{x_i}\mathcal{M}}
\bigl(L_t^{\mathrm{dis}}W(x_i)\bigr)
&=
\frac{1}{t^2}
\sum_{L}
\left(
\sum_{j=1}^{N}
\frac{G_t(x_i,x_j)}
{d_t^{\mathrm{dis}}(x_i)}
\sum_{J}
f_L(x_i)-
f_J(x_j)
\det\!\bigl(A^{L,J}(i,j)\bigr)
\right)
O_L(x_i).
\end{aligned}
\end{equation}

To rewrite \eqref{expansiongrande} in matrix form, we use the representation

\begin{equation*}
   W=\mathbf{O}_k*\mathbf{f} 
\end{equation*}
introduced in Section~\ref{sec:matrixdifarrays}, where $*$ denotes standard matrix multiplication. This observation motivates the following matrix representation of the projected discrete ambient connection Laplacian:
\begin{equation}
\label{matrixambientlaplacian}
\mathbf{P}_{\Lambda^kT\mathcal{M}}
\bigl(L_t^{\mathrm{dis}}W\bigr)
:=
\frac{1}{t^2}
\mathbf{O}_k
*
(\mathbf{Id}-\mathbf{CL})
*
\mathbf{f},
\end{equation}
where $\mathbf{CL}$ and $\mathbf{Id}$ are block matrices of size $N\times N$, whose blocks belong to
$\mathbb{R}^{\binom{d}{k}\times\binom{d}{k}}$. Consequently, both matrices have overall dimension
\[
N\binom{d}{k}\times N\binom{d}{k}.
\]

Let
\[
J_1,\ldots,J_{\binom{d}{k}}
\]
be the ordered collection of all multi-indices of length $k$. Then, for
$i,j=1,\ldots,N$, the $(i,j)$-th block of $\mathbf{CL}$ is
\begin{equation}
\label{eq:blockCL}
\mathbf{CL}(i,j)
=
\frac{G_t(x_i,x_j)}
{d_t^{\mathrm{dis}}(x_i)}
\left[
\det\!\bigl(A^{J_p,J_q}(i,j)\bigr)
\right]_{p,q=1}^{\binom{d}{k}},
\end{equation}
that is,
\[
\mathbf{CL}(i,j)_{pq}
=
\frac{G_t(x_i,x_j)}
{d_t^{\mathrm{dis}}(x_i)}
\det\!\bigl(A^{J_p,J_q}(i,j)\bigr),
\qquad
1\le p,q\le\binom{d}{k}.
\]
Finally, let $\mathbf{Id}$ denote the identity matrix of size
$N\binom{d}{k}\times N\binom{d}{k}$.

Recall that $\mathbf{O}_k$, defined in Equation~\eqref{def:negritaO}, is a
block-diagonal matrix whose diagonal blocks consist of the orthonormal bases of
$\Lambda^kT_{x_i}\mathcal{M}$ written as columns. Consequently,
$\mathbf{O}_k$ is an orthogonal matrix, that is,
\[
\mathbf{O}_k^{T}\mathbf{O}_k=\mathbf{Id}.
\]
Therefore, if
\[
W=\mathbf{O}_k*\mathbf{f},
\]
then
\[
\mathbf{O}_k^{T}W
=
\mathbf{O}_k^{T}\mathbf{O}_k*\mathbf{f}
=
\mathbf{f}.
\]
Substituting this identity into Equation~\eqref{matrixambientlaplacian}, we obtain
\begin{equation}
\label{matrixlaplacianocompleta}
\mathbf{P}_{\Lambda^kT\mathcal{M}}
\bigl(L_t^{\mathrm{dis}}W\bigr)
=
\frac{1}{t^2}
\mathbf{O}_k
*
(\mathbf{Id}-\mathbf{CL})
*
\mathbf{O}_k^{T}W.
\end{equation}
Hence, the matrix representation of the projected discrete ambient connection
Laplacian is given by
\[
\frac{1}{t^2}\,
\mathbf{O}_k
*
(\mathbf{Id}-\mathbf{CL})
*
\mathbf{O}_k^{T}.
\]

Observe that the spectrum of the matrix
$\mathbf{Id}-\mathbf{CL}$ differs from that of $\mathbf{CL}$ only by a
uniform shift of $-1$ in each eigenvalue. Therefore, the spectral
properties of the projected discrete ambient connection Laplacian are
completely determined by the matrix $\mathbf{CL}$. Consequently,
$\mathbf{CL}$ is the matrix that encodes the intrinsic information of the
projected ambient connection Laplacian acting on differential forms.

Finally, we summarize the construction of the matrix $\mathbf{CL}$,
together with the selection of the infinitesimal parameter discussed in
Equation~\eqref{parametroinfinite}, in
Algorithm~\ref{algoHCLMatrix}.

\begin{algorithm}[H]
\caption{Construction of the matrix $\mathbf{CL}$}
\label{algoHCLMatrix}

\textbf{Input:}
A data set
\[
X=\{x_1,x_2,\ldots,x_N\},
\]
the manifold dimension $d$, and an orthonormal basis of
$T_{x_i}\mathcal{M}$ for each $x_i$, $i=1,\ldots,N$.

\begin{enumerate}

\item Compute the infinitesimal parameter
\[
t=\frac{1}{N^{2/(d+6)}}.
\]

\item Initialize $\mathbf{CL}$ as the zero block matrix of size
$N\times N$.

\item For $i=1,\ldots,N$:
    \begin{enumerate}
        \item For $j=1,\ldots,N$:
        \begin{enumerate}
            \item Compute the block
            $\mathbf{CL}(i,j)$ according to
            Equation~\eqref{eq:blockCL}.
        \end{enumerate}
    \end{enumerate}

\item Return the matrix $\mathbf{CL}$.

\end{enumerate}

\end{algorithm}

Finally, we establish that the matrix representation of the proposed
projected ambient connection Laplacian is independent of the particular
choice of local orthonormal bases up to orthogonal similarity.

Let
\[
\{O_{1}(x_i),\ldots,O_{d}(x_i)\}
\quad\text{and}\quad
\{\widehat{O}_{1}(x_i),\ldots,\widehat{O}_{d}(x_i)\}
\]
be two orthonormal bases of $T_{x_i}\mathcal{M}$ for each sample point
$x_i$. Let $\mathbf{CL}$ and $\widehat{\mathbf{CL}}$ denote the matrices
constructed from these two choices of local bases through
Equation~\eqref{matrixambientlaplacian}, and let
$\mathbf{O}_k$ and $\widehat{\mathbf{O}}_k$ be the corresponding block-diagonal
matrices defined in Equation~\eqref{def:negritaO}.

Since both constructions represent the same projected discrete ambient
connection Laplacian, we have
\begin{equation}
\mathbf{P}_{\Lambda^kT\mathcal{M}}
\bigl(L_t^{\mathrm{dis}}W\bigr)
=
\frac{1}{t^2}
\mathbf{O}_k
(\mathbf{Id}-\mathbf{CL})
\mathbf{O}_k^{T}W
=
\frac{1}{t^2}
\widehat{\mathbf{O}}_k
(\mathbf{Id}-\widehat{\mathbf{CL}})
\widehat{\mathbf{O}}_k^{T}W.
\end{equation}

Multiplying on the left by $\mathbf{O}_k^{T}$ and on the right by
$\mathbf{O}_k$, and using the orthogonality of $\mathbf{O}_k$, gives
\begin{equation}
\label{eqexpaorto}
\mathbf{Id}-\mathbf{CL}
=
(\mathbf{O}_k^{T}\widehat{\mathbf{O}}_k)
(\mathbf{Id}-\widehat{\mathbf{CL}})
(\mathbf{O}_k^{T}\widehat{\mathbf{O}}_k)^{T}.
\end{equation}

Define
\[
\mathbf{Q}:=\mathbf{O}_k^{T}\widehat{\mathbf{O}}_k.
\]
Since both $\mathbf{O}_k$ and $\widehat{\mathbf{O}}_k$ are block-diagonal,
$\mathbf{Q}$ is also block-diagonal. Moreover, for each
$i=1,\ldots,N$, the corresponding diagonal blocks of
$\mathbf{O}_k$ and $\widehat{\mathbf{O}}_k$ consist of orthonormal bases
of the same vector space $\Lambda^k(T_{x_i}\mathcal{M})$. Hence, each
diagonal block of $\mathbf{Q}$ is the orthogonal change-of-basis matrix
between these two bases. Therefore, every diagonal block of
$\mathbf{Q}$ is orthogonal, and consequently $\mathbf{Q}$ itself is
orthogonal.

Substituting this identity into Equation~\eqref{eqexpaorto} yields
\[
\mathbf{Id}-\mathbf{CL}
=
\mathbf{Q}
(\mathbf{Id}-\widehat{\mathbf{CL}})
\mathbf{Q}^{T}.
\]
Since the identity matrix is invariant under orthogonal similarity, it
follows that
\begin{equation}
\label{CLsimilaridadorto}
\mathbf{CL}
=
\mathbf{Q}\,
\widehat{\mathbf{CL}}\,
\mathbf{Q}^{T}.
\end{equation}

Therefore, the matrix $\mathbf{CL}$ is uniquely determined up to
orthogonal similarity. Consequently, all quantities invariant under
similarity transformations, including its spectrum, characteristic
polynomial, trace, determinant, and eigenvalue multiplicities, are
independent of the particular choice of local orthonormal bases.
Although the matrix representation depends on the selected bases, these
spectral invariants are intrinsic to the projected ambient connection
Laplacian. In particular, the spectrum of $\mathbf{CL}$ is an intrinsic
geometric invariant of the proposed discrete operator.

\section{Numerical Approximation of the Heat Flow Generated by the Projected Ambient Connection Laplacian}
\label{ecuacioncalor}

In the previous section, we introduced a data-driven matrix representation of
the projected ambient connection Laplacian acting on differential forms.
We now employ this discrete operator to construct a numerical approximation of
the corresponding heat flow on a smooth manifold.

Throughout this section, differential forms are identified with differential
arrays through the canonical correspondence induced by the musical isomorphism
defined in \eqref{eq:sharp-map}. Accordingly, differential forms are used for
the theoretical analysis, whereas differential arrays are employed for the
numerical implementation.

Let $\omega_0$ be a smooth differential $k$-form on $M$. We consider the
initial value problem
\begin{equation}
\label{eq:heat_forms}
\begin{cases}
\dfrac{\partial \omega}{\partial t}
=
-\tilde{\Delta}_{\nabla}\omega,
& (x,t)\in M\times(0,\infty),\\[8pt]
\omega(x,0)=\omega_0(x),
& x\in M,
\end{cases}
\end{equation}
where the projected ambient connection Laplacian
$\tilde{\Delta}_{\nabla}$, introduced in
\eqref{proyeLaplac}, is defined by
\begin{equation}
\label{eq:projected_connection_laplacian}
(\tilde{\Delta}_{\nabla}\omega)(x)
:=
\mathbf{P}_{\Lambda^kT_xM}
\bigl(\Delta_{\nabla}\omega(x)\bigr),
\end{equation}
where $\mathbf{P}_{\Lambda^kT_xM}$ denotes the orthogonal projection onto
$\Lambda^kT_xM$. Under the differential-array representation, the time
derivative is interpreted componentwise.

Our objective is to approximate the solution of
\eqref{eq:heat_forms} directly from a finite point cloud sampling of the
underlying manifold. To this end, we replace the continuous operator
$\tilde{\Delta}_{\nabla}$ by the discrete matrix
\begin{equation}
\label{matrixrepresentaciondiffusion}
\hat{A}
=
\frac{2}{t^2}\,
\mathbf{O}_k
*
(\mathbf{Id}-\mathbf{CL})
*
\mathbf{O}_k^{T},
\end{equation}
where $\mathbf{CL}$ is the data-driven operator introduced in
Section~\ref{matrixrepresentacionC}. By the consistency results established in
Theorems~\eqref{convergencerate}, \eqref{discreteinfgen}, and
\eqref{matrixlaplacianocompleta}, the matrix $\hat{A}$ provides a consistent
approximation of the projected ambient connection Laplacian. Since $\hat{A}$ is
constructed entirely from the sampled data, it yields a fully data-driven
discretization of the heat equation \eqref{eq:heat_forms}, enabling the
numerical approximation of the associated diffusion process directly on the
point cloud.

The kernel bandwidth parameter is chosen according to the asymptotic scaling
introduced in \eqref{parametroinfinite},
\begin{equation}
\label{eq:bandwidth_scaling}
t=N^{-2/(d+6)},
\end{equation}
which is known to provide the optimal convergence rate for graph-based
approximations of Laplace-type operators
\cite{singer2006graph}. 
Given an initial discrete differential $k$-array
$\omega^{(0)}$, we approximate the solution of
\eqref{eq:heat_forms} by the explicit (forward) Euler method. For
$s=0,1,2,\ldots$, we define
\begin{equation}
\label{eq:euler_scheme}
\omega^{(s+1)}
=
\omega^{(s)}
-
\tau_s\,\hat{A}\,\omega^{(s)},
\end{equation}
where $\tau_s>0$ denotes the time step at the $s$-th iteration.
This yields a fully data-driven approximation of the heat flow generated
by the projected ambient connection Laplacian on differential forms.

Since the discrete operator $\hat{A}$ is generally non-symmetric, the
classical CFL condition based on the largest eigenvalue is no longer
applicable. Instead, we employ the following norm-based sufficient
stability criterion:
\begin{equation}
\label{eq:tau_condition}
\tau_s
\le
\frac{1}{\|\hat{A}\|_2},
\end{equation}
where $\|\hat{A}\|_2$ denotes the spectral norm of $\hat{A}$.
This condition follows from the bound
$\rho(\hat{A})\le\|\hat{A}\|_2$ and guarantees that the explicit Euler
iteration remains contractive in the induced matrix norm. Although more
conservative than the exact spectral stability restriction, this choice
is inexpensive to evaluate and can be computed efficiently using
standard numerical linear algebra routines, making it particularly
suitable for large-scale data-driven problems
\cite{higham2008functions,trefethen2005spectra,moler2003nineteen}.

\subsection{Numerical Experiment}

We now present a numerical experiment illustrating the heat flow generated by
the projected ambient connection Laplacian. All computations were performed in
\textsc{MATLAB R2017b} on a laptop equipped with an Intel Core i5-1235U
processor (1.30~GHz) and 8~GB of RAM. The implementation is publicly available
in the GitHub repository \cite{codigoVF}.

The experiment is conducted on the two-dimensional unit sphere
\[
    S^2=\{x\in\mathbb{R}^3:\|x\|=1\},
\]
and considers the case $k=1$, where differential $1$-forms are identified with
tangent vector fields. Starting from a prescribed tangent vector field, we
numerically approximate its evolution under the heat flow induced by the
projected ambient connection Laplacian.

For visualization purposes, the computed solution is displayed both on the
embedded sphere and through the parametrization
$\Phi:[0,1]\times[0,1]\rightarrow S^2$ given by
\begin{equation}
\label{eq:param}
    \Phi(u,v)=
    \bigl(
        \sin(\pi u)\cos(2\pi v),\;
        \sin(\pi u)\sin(2\pi v),\;
        \cos(\pi u)
    \bigr),
    \qquad
    (u,v)\in[0,1]\times[0,1],
\end{equation}
where $u$ and $v$ denote the polar and azimuthal coordinates, respectively.
Each figure shows the pullback of the solution through $\Phi$, together with
its pointwise Euclidean norm evaluated at the sampled points.

Throughout the experiment, we generate $N=2000$ points independently and
uniformly distributed on $S^2$. The numerical solution of the heat equation is
computed using the iterative scheme introduced in \cref{ecuacioncalor}, with
time step
\begin{equation}
\label{eq:timestep}
    \tau_s=
   \frac{0.9}{\|\hat{A}\|_2},
\end{equation}
which satisfies the stability condition established in
\cref{eq:tau_condition}. Therefore, the numerical scheme is stable and convergent
throughout the experiment.

As the initial condition, we consider the tangent vector field
\begin{equation}
\label{eq:v1}
v_1^0(p)
=
(1,1,1)-\langle (1,1,1),p\rangle\,p,
\qquad p\in S^2,
\end{equation}
illustrated in Figure~\ref{fig:initial_vector_field}. The vector field
$v_1^0$ is obtained by orthogonally projecting the constant ambient vector
$(1,1,1)\in\mathbb{R}^3$ onto the tangent space $T_pS^2$ at each point
$p\in S^2$. Consequently, $v_1^0(p)\in T_pS^2$ for every $p\in S^2$.

\begin{figure}[htbp]
    \centering
    \includegraphics[width=\textwidth]{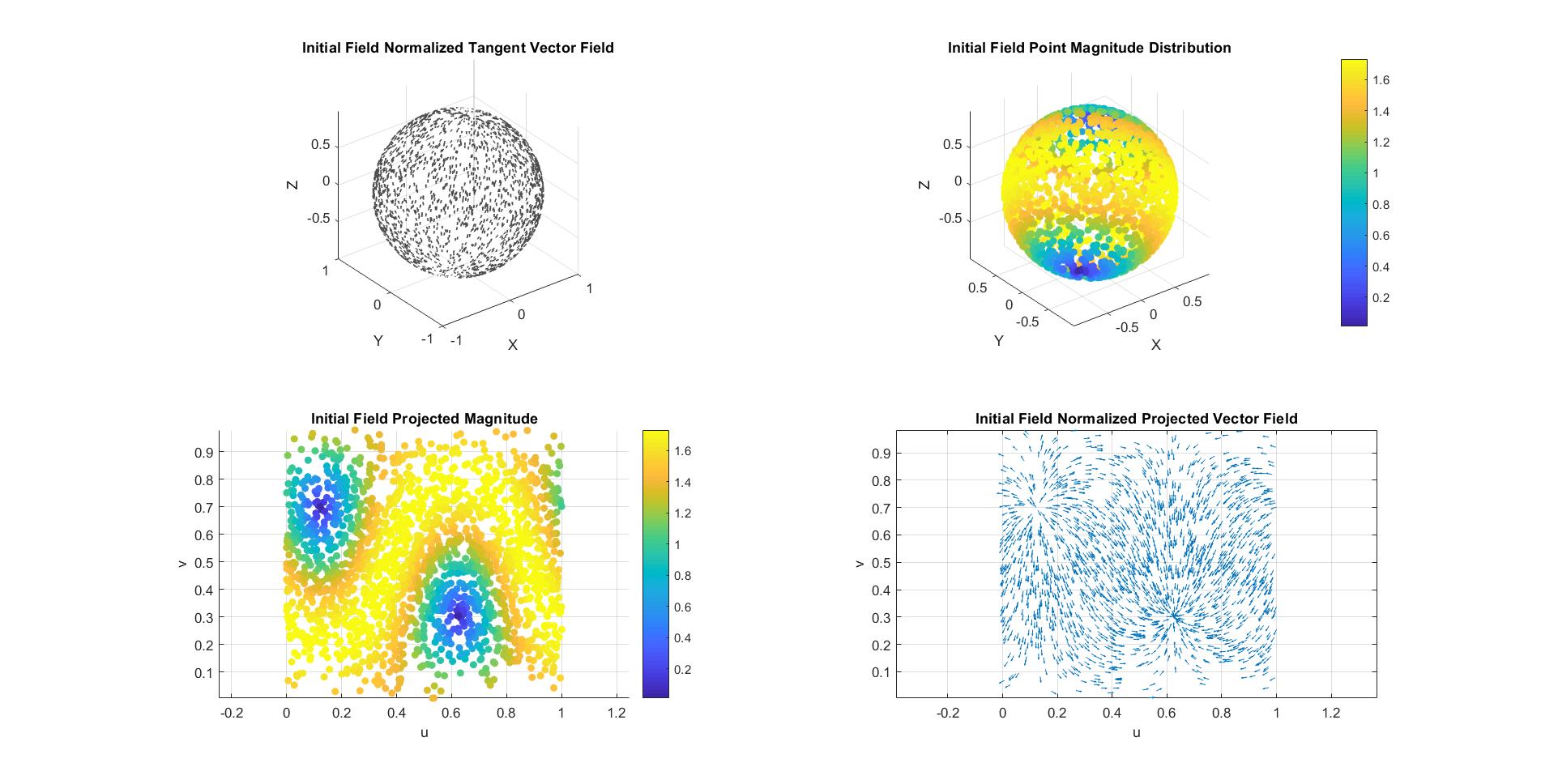}
    \caption{Initial vector field $v_1^0$. The figure shows the normalized
    tangent vector field together with its pointwise magnitude, displayed on
    both the three-dimensional sphere and its two-dimensional polar-coordinate
    parametrization $\Phi$.}
    \label{fig:initial_vector_field}
\end{figure}

Figure~\ref{fig:iterations} displays the pointwise magnitude of the tangent
vector field obtained after $s=10,20,30,$ and $40$ iterations of the proposed
iterative scheme. The corresponding evolution of the discrete $L^2$ norm is
shown in Figure~\ref{fig:iterationsL2norm}.

Both figures clearly indicate that the numerical solution converges to the
zero vector field as the number of iterations increases. This behavior is in
excellent agreement with the analytical solution of the continuous heat
equation. Indeed, the initial tangent vector field $v_1^0$ is an eigenvector
of the projected ambient connection Laplacian. A straightforward componentwise
computation of the ambient connection Laplacian gives
\begin{equation*}
    \Delta_{\bar{\nabla}}v_1^0
    =
    -2(1,1,1)
    +
    6\langle(1,1,1),p\rangle p.
\end{equation*}
Projecting this vector field onto the tangent space $T_p\mathbb{S}^2$ yields
\begin{equation*}
    (\widetilde{\Delta}_{\nabla}v_1^0)(p)
    :=
    \mathbf{P}_{T_p\mathbb{S}^2}
    \bigl(
    \Delta_{\bar{\nabla}}v_1^0(p)
    \bigr)
    =
    -2v_1^0(p),
\end{equation*}
which shows that $v_1^0$ is an eigenvector of the projected ambient connection
Laplacian with eigenvalue $-2$. Consequently, the solution of the vector heat
equation
\[
    \partial_t v=-\widetilde{\Delta}_{\nabla}v,
    \qquad
    v(\cdot,0)=v_1^0,
\]
is given explicitly by (see, e.g., \cite{rosenberg1997laplacian})
\begin{equation}
\label{eq:heat_solution}
    v(t)=e^{-2t}v_1^0.
\end{equation}
Hence,
\[
    \lim_{t\rightarrow\infty}v(t)=0,
\]
that is, the solution converges to the zero vector field.
Therefore, the numerical experiments confirm that the proposed explicit Euler
scheme accurately reproduces the predicted exponential decay, thereby
demonstrating its consistency with the analytical solution of the heat
equation.

\begin{figure}[htbp]
    \centering
    \includegraphics[width=\textwidth]{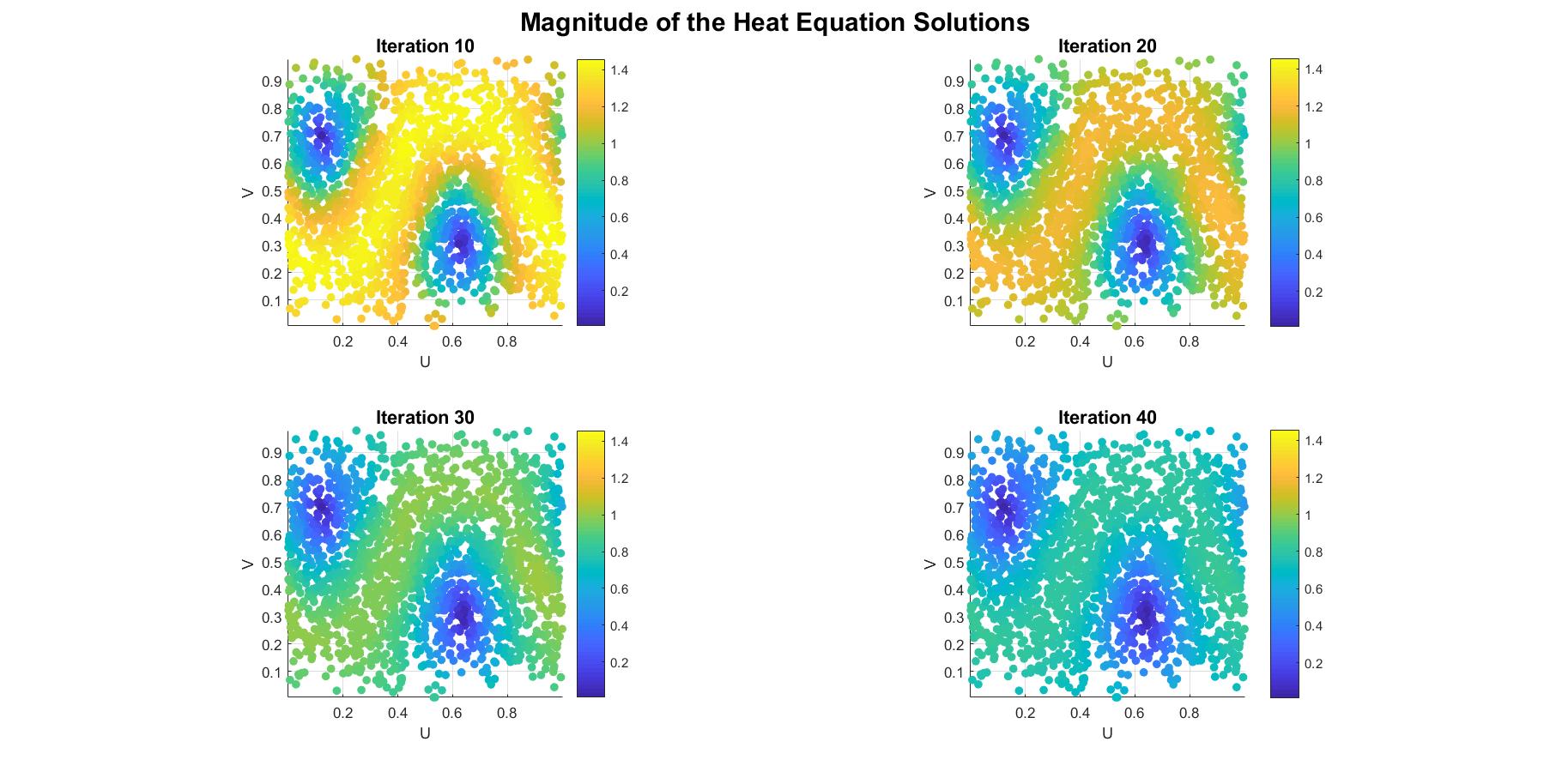}
    \caption{Pointwise magnitude of the numerical approximation of the heat
    flow obtained with the explicit Euler scheme \eqref{eq:euler_scheme} after
    $s=10,20,30,$ and $40$ iterations. The tangent vector field is visualized
    through the two-dimensional parametrization $\Phi$ of the unit sphere.
    The progressive decrease in the magnitude illustrates the convergence of
    the numerical solution toward the zero vector field.}
    \label{fig:iterations}
\end{figure}

\begin{figure}[htbp]
    \centering
    \includegraphics[width=\textwidth]{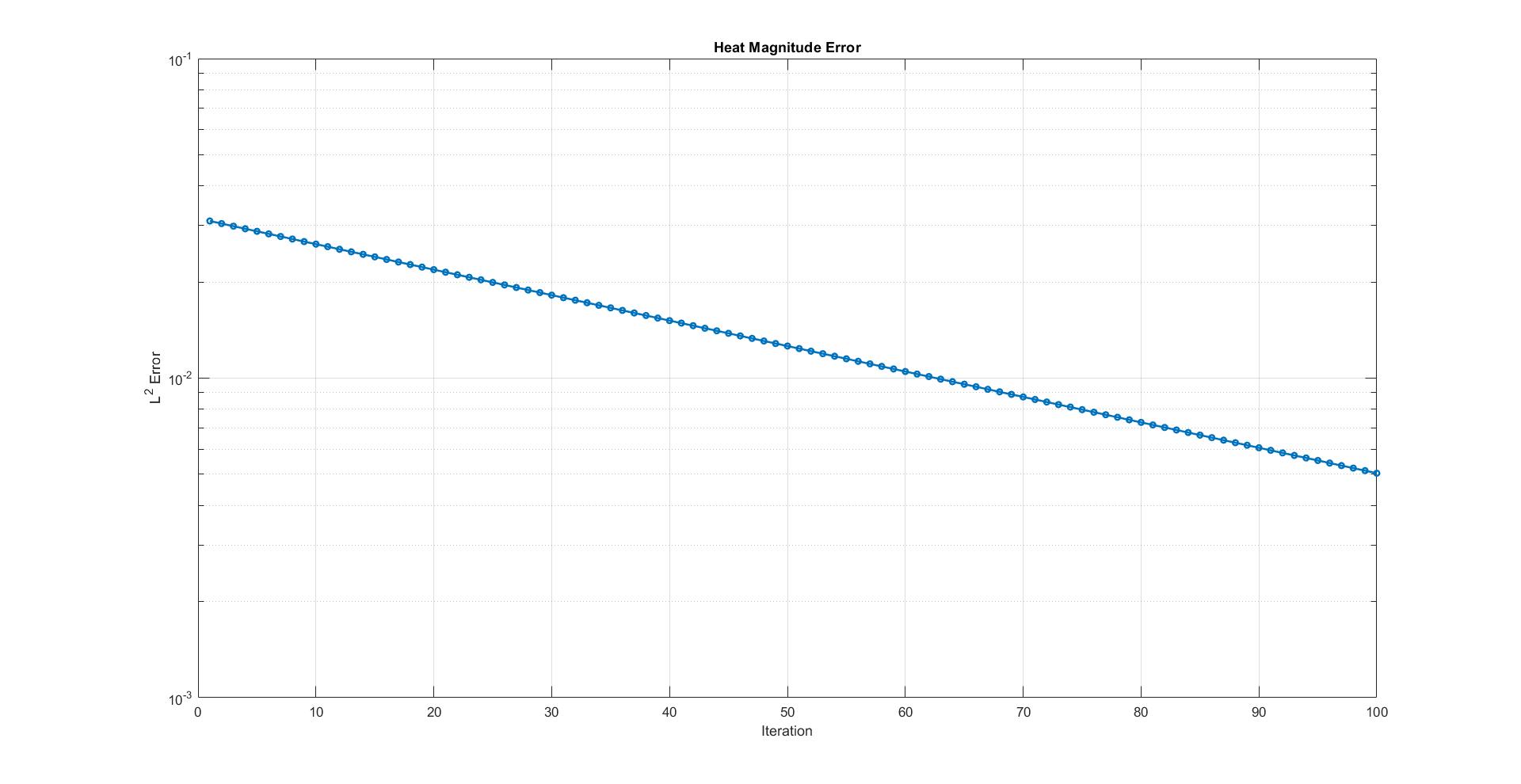}
    \caption{Evolution of the discrete $L^2$ norm of the numerical solution
    corresponding to the initial vector field $v_1^0$ during the first $100$
    Euler iterations. The vertical axis is displayed on a logarithmic scale,
    revealing the expected decay predicted by the analytical
    solution of the heat equation.}
    \label{fig:iterationsL2norm}
\end{figure}

\section{Conclusions and Future Work}

In this work, we introduced a data-driven discretization framework for
differential forms on smooth manifolds. Through the musical isomorphism,
differential forms are identified with alternating differential arrays,
providing a representation that transforms these geometric objects into
computationally tractable quantities suitable for numerical analysis.

From a computational perspective, we investigated the ambient connection
Laplacian, which is defined by applying the Laplace--Beltrami operator
componentwise to vector-valued functions. By projecting this operator onto
the discrete bundle of differential forms, we obtain a linear operator that
maps discrete differential forms into discrete differential forms. This
projected ambient connection Laplacian is closely related to the Hodge
Laplacian, differing only by lower-order perturbation terms determined by
the geometry of the underlying manifold.

A principal contribution of this work is the construction of a matrix-based,
data-driven approximation of the projected ambient connection Laplacian.
Building upon the classical theory of diffusion maps, we derive a discrete
operator that approximates the continuous projected ambient connection
Laplacian directly from point cloud data. This approximation provides a
practical numerical framework for simulating diffusion processes on
differential forms without requiring a mesh or explicit knowledge of the
manifold geometry.

Based on the proposed discretization, we developed an explicit Euler scheme
for the numerical approximation of the heat equation on differential forms.
The numerical experiments demonstrate that the proposed method accurately
reproduces the theoretical exponential decay predicted by the continuous
model, thereby confirming the consistency of the discrete approximation and
its effectiveness in capturing the diffusion dynamics of differential forms.

Future research will focus on extending the proposed framework to more
general geometric evolution equations involving differential forms and
vector fields. In particular, the proposed methodology provides a promising
foundation for the numerical simulation of partial differential equations
governed by diffusion operators on manifolds, as well as for applications in
computational differential geometry, geometric data analysis, and
data-driven scientific computing.

\bigskip
\noindent\textbf{Acknowledgements.} 
The first author was supported by the \textbf{Centro de Modelamiento Matemático (CMM)} through the BASAL grant FB210005 for Centers of Excellence from ANID-Chile. The second author was supported by the ANID Fondecyt Postdoctoral Grant No.~3220631.


\bibliographystyle{apalike}
\bibliography{bibli}


\begin{appendices}

\section{Proof of  Proposition \ref{prop:musical-generalization}}
\label{proofprop1}

\begin{proof}
\textbf{Existence.}
Since $\omega \colon V^k \to \mathbb{R}$ is multilinear, the universal
property of the tensor product
\cite[Prop.~12.7]{lee2012introduction}
yields a unique linear functional
\[
    \widetilde{\omega} \colon V^{\otimes k} \to \mathbb{R}
\]
such that
\[
    \widetilde{\omega}(v_1 \otimes \cdots \otimes v_k)
    =
    \omega(v_1,\dots,v_k)
\]
for all $v_1,\dots,v_k \in V$.

Because $V$ is finite dimensional, the tensor space
$V^{\otimes k}$ is also finite dimensional. Endow
$V^{\otimes k}$ with the Frobenius inner product
$\langle \cdot,\cdot \rangle_F$.
By the Riesz representation theorem, there exists a unique tensor
$W \in V^{\otimes k}$ such that
\[
    \widetilde{\omega}(T)
    =
    \langle W,T\rangle_F
\]
for every $T \in V^{\otimes k}$.
In particular,
\[
    \omega(v_1,\dots,v_k)
    =
    \langle W,\,
    v_1 \otimes \cdots \otimes v_k
    \rangle_F .
\]

It remains to show that $W$ is alternating.
Let $\sigma \in S_k$, and define the permuted tensor
$W^\sigma \in V^{\otimes k}$ by
\[
    W^\sigma(i_1,\dots,i_k)
    :=
    W(i_{\sigma(1)},\dots,i_{\sigma(k)}).
\]
Then, for arbitrary vectors $v_1,\dots,v_k \in V$, we have
\begin{align*}
    \langle W^\sigma,\,
    v_1 \otimes \cdots \otimes v_k
    \rangle_F
    &=
    \langle W,\,
    v_{\sigma(1)} \otimes \cdots \otimes v_{\sigma(k)}
    \rangle_F \\
    &=
    \omega(v_{\sigma(1)},\dots,v_{\sigma(k)}) \\
    &=
    \operatorname{sgn}(\sigma)\,
    \omega(v_1,\dots,v_k) \\
    &=
    \left\langle
        \operatorname{sgn}(\sigma)\, W,\,
        v_1 \otimes \cdots \otimes v_k
    \right\rangle_F .
\end{align*}
Since decomposable tensors span $V^{\otimes k}$, the above identity implies
that
\[
    \langle W^\sigma,T\rangle_F
    =
    \left\langle
        \operatorname{sgn}(\sigma)\,W,
        T
    \right\rangle_F
\]
for every $T \in V^{\otimes k}$.
By the uniqueness of the Riesz representative, we conclude that
\[
    W^\sigma
    =
    \operatorname{sgn}(\sigma)\,W .
\]
Hence $W$ is alternating.

\medskip

\textbf{Uniqueness.}
Suppose that $W_1,W_2 \in \Lambda^k V$ satisfy
\[
    \omega(v_1,\dots,v_k)
    =
    \langle W_j,\,
    v_1 \otimes \cdots \otimes v_k
    \rangle_F,
    \qquad j=1,2,
\]
for all $v_1,\dots,v_k \in V$.
Then
\[
    \langle W_1-W_2,\,
    v_1 \otimes \cdots \otimes v_k
    \rangle_F
    =
    0
\]
for all $v_1,\dots,v_k \in V$.
Since decomposable tensors span $V^{\otimes k}$, it follows that
\[
    \langle W_1-W_2,T\rangle_F = 0
\]
for all $T \in V^{\otimes k}$.
Therefore $W_1-W_2=0$, and consequently $W_1=W_2$.
\end{proof}

\end{appendices}

\end{document}